\documentclass[12pt,reqno]{amsart}
\usepackage{amsmath}%,cancel,fullpage}
\usepackage{tikz}
\usepackage[all]{xy}
\usetikzlibrary{decorations.markings}
\usepackage{color}
\usepackage{amsthm}
\usepackage{amsfonts}
\usepackage{amssymb}
\usepackage{setspace}
\usepackage{youngtab}
\usepackage{ytableau}
\usepackage{empheq}
\usepackage[bookmarks,colorlinks,breaklinks]{hyperref}  % PDF hyperlinks, with coloured links
\hypersetup{linkcolor=blue,citecolor=blue,filecolor=blue,urlcolor=blue} % all blue links
\theoremstyle{plain}

\definecolor{gr40}{gray}{0.40}

\newcommand{\beqn}{\begin{eqnarray}}
\newcommand{\eeqn}{\end{eqnarray}}

\newtheorem{Theorem}{Theorem}[section]
\newtheorem{Proposition}[Theorem]{Proposition}
\newtheorem{Definition}[Theorem]{Definition}
\newtheorem{Example}[Theorem]{Example}
\newtheorem{Corollary}[Theorem]{Corollary}
\newtheorem{Lemma}[Theorem]{Lemma}
\newtheorem{Remark}[Theorem]{Remark}

\newcommand{\asc}{\mathrm{asc}}
\newcommand{\supp}{\mathrm{supp}}
\newcommand{\arm}{\mathrm{arm}}
\newcommand{\leg}{\mathrm{leg}}
\newcommand{\comp}{\mathrm{comp}}
\newcommand{\set}{\mathrm{set}}
\newcommand{\des}{\mathrm{Des}}
\newcommand{\ctau}{\tau}
\newcommand{\ncsa}{\mathbf{s}_{\alpha}}
\newcommand{\ncsb}{\mathbf{s}_{\beta}}
\newcommand{\ncsg}{\mathbf{s}_{\gamma}}
\newcommand{\nch}{\mathbf{h}}% noncommutative h-basis for NSym
\newcommand{\rib}{\mathbf{r}} % ribbon basis for NSym
\newcommand{\ncs}{\mathbf{s}} %noncommutative schur
\newcommand{\upl}{\mathfrak{t}}%up operator for left pieri poset
\definecolor{myblue}{rgb}{0.4, 0.5, 0.7}
\newcommand{\myb}{*(myblue)}
\newcommand{\cw}{*(white)}

\begin{document}
\title{A Murnaghan-Nakayama Rule For Noncommutative Schur Functions}
\author{VASU V. TEWARI}
\address{
Vasu V. Tewari\\
Department of Mathematics\\
University of British Columbia\\
Vancouver, BC V6T 1Z2\\
Canada
}
\email{vasu@math.ubc.ca}

\begin{abstract}
We prove a Murnaghan-Nakayama rule for the noncommutative Schur functions introduced by Bessenrodt, Luoto and van Willigenburg. In other words, we give an explicit combinatorial formula for expanding the product of a noncommutative power sum symmetric function and a noncommutative Schur function in terms of noncommutative Schur functions. In direct analogy to the classical Murnaghan-Nakayama rule, the summands are computed using a noncommutative analogue of border strips, and have coefficients $\pm 1$ determined by the height of these border strips. The rule is proved by interpreting the noncommutative Pieri rules for noncommutative Schur functions in terms of box-adding operators on compositions.
\end{abstract}
\maketitle

\section{Introduction}
The Murnaghan-Nakayama rule \cite{Murnaghan,Nakayama} is a combinatorial procedure to compute the character table of the symmetric group. The combinatorial objects involved, that aid the said computation, are called border strips. An alternate description of the Murnaghan-Nakayama rule is that it gives an explicit combinatorial description of the expansion of the product of a power sum symmetric function $p_k$, where $k$ is a positive integer, and a Schur function $s_{\lambda}$ as a sum of Schur functions in the algebra of symmetric functions
\begin{eqnarray}\label{eqn:classicalMurNak} 
p_k\cdot s_{\lambda}=\displaystyle\sum_{\mu}(-1)^{ht(\mu/\lambda)}s_{\mu},
\end{eqnarray}
where the sum is over all partitions $\mu$ such that the skew shape $\mu/\lambda$ is a border strip of size $k$ and $ht(\mu/\lambda)$ is a certain statistic associated with the border strip $\mu/\lambda$.

In this article, we consider an analogue in the algebra of noncommutative symmetric functions. The role played by the Schur functions in the classical setting is taken by the noncommutative Schur functions introduced in \cite{BLvW}, that are dual to the quasisymmetric Schur functions \cite{HLMvW-1} arising from the combinatorics of Macdonald polynomials \cite{HHL-1}. To be more precise, we expand the product $\Psi_{r}\cdot \ncsa$ in the basis of noncommutative Schur functions. Here $\Psi_r$ denotes the noncommutative power sum symmetric function of the first kind, indexed by a positive integer $r$, while $\ncsa$ denotes the noncommutative Schur function indexed by a composition $\alpha$. The main theorem we prove is that, much like \eqref{eqn:classicalMurNak}, we have an expansion
\begin{eqnarray}\label{eqn:ncMurnak}
\Psi_{r}\cdot \ncsa=\displaystyle\sum_{\beta}(-1)^{ht (\beta/\!\!/\alpha)}\ncsb ,
\end{eqnarray}
where the sum is over certain compositions $\beta$ such that the skew reverse composition shape $\beta/\!\!/\alpha$ gives rise to a noncommutative analogue of a border strip of size $r$ and $ht (\beta/\!\!/\alpha)$ is a certain statistic similar to $ht(\mu/\lambda)$ in \eqref{eqn:classicalMurNak}.

The organization of this article is as follows. In Section \ref{sec:backgroundsym}, we introduce all the notation and definitions necessary for stating the classical Murnaghan-Nakayama rule. Section \ref{sec:backgroundNsym} introduces the algebra of noncommutative symmetric functions and the distinguished basis of noncommutative Schur functions amongst other bases. Furthermore, we introduce box-adding operators reminiscent of those in \cite{BSZ,fomin-greene-1} that will be fundamental to our proofs. The goal of Section \ref{sec:boxaddingmurnaghannakayama} is to give a rudimentary version of the noncommutative Murnaghan-Nakayama rule in terms of box-adding operators in Equation \eqref{advMurNak}. Finally in Section \ref{sec:finalmurnaghannakayama}, we give the noncommutative Murnaghan-Nakayama rule that mirrors the classical version in Theorem \ref{MurNakfinal}.

\subsection*{Acknowledgement}
The author is grateful to Stephanie van Willigenburg for suggesting the problem and for illuminating discussions.
\section{Background on symmetric functions}\label{sec:backgroundsym}
\subsection{Partitions}
We will start by defining some of the combinatorial structures that we will be encountering. All the notions introduced in this section are covered in more detail in \cite{macdonald-1,stanley-ec2,sagan}. Our first definition has to do with the notion of partition.
\begin{Definition}\label{def:partition}
A \emph{partition} $\lambda$ is a finite list of positive integers $(\lambda_1,\ldots,\lambda_k)$ satisfying $\lambda_1\geq \lambda_2\geq \cdots\geq \lambda_k$. The integers appearing in the list are called the \emph{parts} of the partition.
\end{Definition}
Given a partition $\lambda=(\lambda_1,\ldots,\lambda_k)$, the \textit{size} $\lvert \lambda\rvert$ is defined to be $\sum_{i=1}^{k}\lambda_i$. The number of parts of $\lambda$ is called the \textit{length}, and is denoted by $l(\lambda)$. If $\lambda$ is a partition satisfying $\lvert \lambda\rvert=n$, then we write it as $\lambda \vdash n$. By convention, there is a unique partition of size and length $0$, and we denote it by $\varnothing$.

We will be depicting a partition using its \textit{Ferrers diagram} (or \textit{Young diagram}). Given a partition $\lambda=(\lambda_1,\ldots,\lambda_k)\vdash n$, the Ferrers diagram of $\lambda$, also denoted by $\lambda$, is the left-justified array of $n$ boxes, with $\lambda_i$ boxes in the $i$-th row. We will be using the French convention, i.e. the rows are numbered from bottom to top and the columns from left to right. We refer to the box in the $i$-th row and $j$-th column by the ordered pair $(i,j)$. If $\lambda$ and $\mu$ are partitions such that $\mu \subseteq \lambda$, i.e., $l(\mu) \leq l(\lambda)$ and  $\mu_{i} \leq \lambda_{i}$ for all $i=1,2,\ldots,l(\mu)$, then the \textit{skew shape} $\lambda / \mu$ is obtained by removing the first $\mu_{i}$ boxes from the $i$-th row of the Ferrers diagram of $\lambda$ for $1\leq i\leq l(\mu)$.  The \textit{size} of the skew shape $\lambda/\mu$ is equal to the number of boxes in the skew shape, i.e. $\lvert \lambda \rvert-\lvert \mu \rvert$.

\begin{Example}
The Ferrers diagram for the partition $\lambda=(4,3,3,1)\vdash 11$ is shown below.
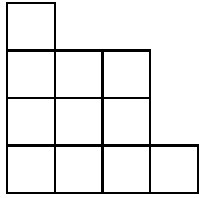
\begin{eqnarray*}
\ydiagram{1,3,3,4}
\end{eqnarray*}
\end{Example}
\subsection{Semistandard reverse tableaux}
In this subsection, we will introduce some classical objects that play a central role in the theory of symmetric functions.
\begin{Definition}\label{def:SSRT}
Given a partition $\lambda$, a \emph{semistandard reverse tableau (SSRT)} $T$ of \emph{shape} $\lambda$ is a filling of the boxes of $\lambda$ with positive integers, satisfying the condition that the entries in $T$ are weakly decreasing along each row read from left to right and strictly decreasing along each column read from bottom to top.
\end{Definition}
A \textit{standard reverse tableau (SRT)} $T$ of shape $\lambda \vdash n$ is an SSRT that contains every positive integer in $[n]=\{1,2,\ldots,n\}$ exactly once. We will denote the set of all SSRTs of shape $\lambda$ by $SSRT(\lambda)$. Given an SSRT $T$, the entry in box $(i,j)$ is denoted by $T_{(i,j)}$.

%%An SSRT of shape $\lambda / \mu$ is a skew shape $\lambda / \mu$ that has boxes filled with positive integers so that the rows are weakly decreasing from left to right and the columns are strictly decreasing from bottom to top.  

Let $\{x_1,x_2,\ldots\}$ be an alphabet comprising of countably many commuting indeterminates $x_1,x_2,\ldots$. Now, given any SSRT $T$ of shape $\lambda\vdash n$, we can associate a monomial $x^{T}$ with it as follows.
\begin{eqnarray*}
x^{T}=\prod_{(i,j)\in\lambda}x_{T_{(i,j)}}
\end{eqnarray*}
\begin{Example}
Shown below are an SSRT $T$ of shape $(4,3,3,1)$ and its associated monomial.
\begin{eqnarray*}
T=\ytableausetup{mathmode}
\begin{ytableau}
1\\3&3&2\\6&5&4\\7&7&6&3
\end{ytableau} \hspace{6mm} 
x^{T}=x_7^2x_6^2x_5x_4x_3^3x_2x_1
\end{eqnarray*}
\end{Example}

\subsection{Symmetric functions}
The algebra of symmetric functions, denoted by $\Lambda$, is the algebra freely generated over $\mathbb{Q}$ by countably many commuting variables $\{h_1,h_2,\ldots\}$. Assigning degree $i$ to $h_i$ (and then extending this multiplicatively) allows us to endow $\Lambda$ with a structure of a graded algebra. A basis for the degree $n$ component of $\Lambda$, denoted by $\Lambda^{n}$, is given by the \textit{complete homogeneous symmetric functions} of degree $n$, $$\{h_{\lambda}=h_{\lambda_1}\cdots h_{\lambda_k}: \lambda=(\lambda_1,\ldots,\lambda_k)\vdash  n\}.$$

A concrete realization of $\Lambda$ is obtained by embedding $\Lambda=\mathbb{Q}[h_1,h_2,\ldots]$ in $\mathbb{Q}[[x_1,x_2,\ldots]]$, i.e. the ring of formal power series in countably many commuting indeterminates $\{x_1,x_2,\ldots\}$, under the identification (extended multiplicatively)
$$h_i \longmapsto \text{ sum of all distinct monomials in $x_1,x_2,\ldots$ of degree $i$}.$$
This viewpoint allows us to think of symmetric functions as being formal power series $f$ in the $x$ variables with the property that $f(x_{\pi(1)},x_{\pi(2)},\ldots)=f(x_{1},x_{2},\ldots)$ for every permutation $\pi$ of the positive integers $\mathbb{N}$.

The \textit{Schur function} indexed by the partition $\lambda$, $s_{\lambda}$, is defined combinatorially as
\begin{eqnarray*}
s_{\lambda}=\displaystyle\sum_{T\in SSRT(\lambda)}x^{T}.
\end{eqnarray*}
Though not evident from the definition above, $s_{\lambda}$ is a symmetric function. Furthermore, the elements of the set $\{s_{\lambda}:\lambda\vdash n\}$ form a basis of $\Lambda ^{n}$ for any positive integer $n$.

Another important class of symmetric functions is given by the \textit{power sum symmetric functions}. The power sum symmetric function, $p_k$ for $k\geq 1$ is defined as 
\begin{eqnarray*}
p_k=\displaystyle\sum_{i\geq 1}x_i^{k}.
\end{eqnarray*}

The classical Murnaghan-Nakayama rule gives an algorithm to compute the product $p_k\cdot s_{\lambda}$ in terms of Schur functions.
To state the rule, we need the notion of a \textit{border strip}. A skew shape $\mu/\lambda$ is called a border strip if it is connected and contains no $2\times 2$ array of boxes. The \textit{height} of a border strip $\mu/\lambda$, denoted by $ht(\mu/\lambda)$, is defined to be one less than the number of rows occupied by the border strip. This given, we have the following.
\begin{Theorem}[Murnaghan-Nakayama rule]
Given a positive integer $k$ and $\lambda\vdash n$, we have
\begin{eqnarray*}
p_{k}\cdot s_{\lambda}&=&\displaystyle\sum_{\mu\vdash |\lambda|+k}(-1)^{ht(\mu/\lambda)}s_{\mu},
\end{eqnarray*}
where the sum is over all partitions $\mu$ such that $\mu/\lambda$ is a border strip of size $k$.
\end{Theorem}
As an aid to understanding the theorem above, we will do an example.
\begin{Example}
Consider the computation of $p_3\cdot s_{(2,1)}$. First we need to find all partitions $\mu$ such that $\mu/(2,1)$ is a border strip of size $3$. All such partitions $\mu$ are listed below.
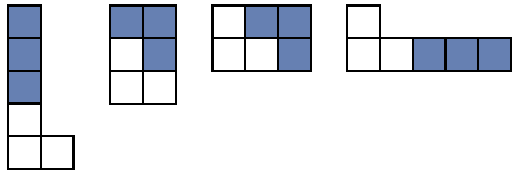
\begin{eqnarray*}
\ytableausetup{mathmode,boxsize=1em}
\begin{ytableau}
 \myb\\\myb\\ \myb\\\cw\\\cw &\cw
\end{ytableau}\hspace{5mm} 
\begin{ytableau}
\myb &\myb\\ \cw &\myb\\ \cw &\cw  
\end{ytableau}\hspace{5mm}
\begin{ytableau}
\cw & \myb &\myb\\\cw &\cw &\myb
\end{ytableau}\hspace{5mm}
\begin{ytableau}
\cw\\\cw & \cw &\myb &\myb &\myb
\end{ytableau}
\end{eqnarray*}
The statistic $ht(\mu/(2,1))$ for the partitions above from left to right is $2,1,1,0$ respectively. Hence, the Murnaghan-Nakayama rule implies that
$$p_3\cdot s_{(2,1)}=s_{(2,1,1,1,1)}-s_{(2,2,2)}-s_{(3,3)}+s_{(5,1)}.$$
\end{Example}

\section{Background on noncommutative symmetric functions}\label{sec:backgroundNsym}
\subsection{Compositions}
\begin{Definition}
 A \emph{composition} $\alpha$ is a finite ordered list of positive integers. The integers appearing in the list are called the \emph{parts} of the composition.
\end{Definition}
Given a composition $\alpha=(\alpha_1,\ldots,\alpha_k)$, the \textit{size} $\lvert \alpha\rvert$ is defined to be $\sum_{i=1}^{k}\alpha_i$. The number of parts of $\alpha$ is called the \textit{length}, and is denoted by $l(\alpha)$. If $\alpha$ is a composition satisfying $\lvert \alpha\rvert=n$, then we write it as $\alpha \vDash n$. By convention, there is a unique composition of size and length $0$, and we denote it by $\varnothing$.

We will associate a \textit{reverse composition diagram} to a composition as follows. Given a composition $\alpha=(\alpha_1,\ldots, \alpha_k)\vDash n$, the \textit{reverse composition diagram} of $\alpha$, also denoted by $\alpha$, is the left-justified array of $n$ boxes with $\alpha_i$ boxes in the $i$-th row. Here we follow the English convention, i.e. the rows are numbered from top to bottom, and the columns from left to right. Again, we refer to the box in the $i$-th row and $j$-th column by the ordered pair $(i,j)$.

Recall now the bijection between compositions of $n$ and subsets of $[n-1]$. Given a composition $\alpha=(\alpha_1,\ldots,\alpha_k)\vDash n$, we can associate a subset of $[n-1]$, called $\set (\alpha)$, by defining it to be $\{\alpha_1,\alpha_1+\alpha_2,\ldots,\alpha_1+\cdots +\alpha_{k-1}\}$. In the opposite direction, given a set $S=\{i_1< \cdots  <i_j\}\subseteq [n-1]$, we can associate a composition of $n$, called $\comp (S)$, by defining it to be $(i_1,i_2-i_1,\ldots, i_j-i_{j-1},n-i_j)$.

Finally, we define the refinement order on compositions. Given compositions $\alpha$ and $\beta$, we say that $\alpha \succcurlyeq \beta$ if one obtains parts of $\alpha$ in order by adding together adjacent parts of $\beta$ in order. The composition $\beta$ is said to be a \textit{refinement} of $\alpha$. 

\begin{Example}
Let $\alpha=(4,2,3)\vDash 9$. Then $\set(\alpha)=\{4,6\}\subseteq [8]$. Shown below is the reverse composition diagram of $\alpha$.
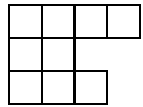
\begin{eqnarray*}
\ydiagram{4,2,3}
\end{eqnarray*}
Notice also that $(4,2,3)\succcurlyeq (3,1,2,1,2)$.
\end{Example}

\subsection{Semistandard reverse composition tableaux}
Let $\beta = (\beta_1,\ldots, \beta_l)$ and  $\alpha$ be compositions. Define a cover relation, $\lessdot_{c}$, on compositions in the following fashion.
\begin{eqnarray*}
\beta \lessdot_{c} \alpha \text{ iff }\left\lbrace \begin{array}{ll} \alpha= (1,\beta_1,\ldots,\beta_l) & \text{ or}\\ \alpha= (\beta_1,\ldots,\beta_k +1,\ldots,\beta_l) & \text{ and $\beta_i\neq \beta_k$ for all $i<k$.} \end{array}\right.
\end{eqnarray*}
The \emph{reverse composition poset} $\mathcal{L}_{c}$ is the poset on compositions where the partial order $<_{c}$ is obtained by taking the transitive closure of the cover relations above.
If $\beta <_{c} \alpha$, the \textit{skew reverse composition shape} $\alpha/\!\!/\beta$ is defined to be 
$$
\alpha/\!\!/\beta = \{(i,j): (i,j)\in \alpha, (i,j)\notin \beta \}.
$$
Here $\alpha$ is called the \textit{outer shape} and $\beta$ the \textit{inner shape}. If $\beta=\varnothing$, instead of writing $\alpha/\!\!/\varnothing$, we just write $\alpha$. 

If $\alpha/\!\!/\beta$ does not have two boxes belonging to the same column, then it is called a \textit{horizontal strip}, while if it does not have two boxes lying in the same row it is called a \textit{vertical strip}.

\begin{Definition}
A \emph{semistandard reverse composition tableau} (SSRCT) $\tau$ of shape $\alpha /\!\! /\beta$ is a filling 
\begin{eqnarray*}
\tau: \alpha/\!\!/ \beta \longrightarrow \mathbb{Z}^{+}
\end{eqnarray*}
that satisfies the following conditions
\begin{enumerate}
\item the rows are weakly decreasing from left to right
\item the entries in the first column are strictly increasing from top to bottom
\item if $i< j$ and $(j,k+1) \in \alpha/\!\! /\beta$ and either $(i,k)\in \beta$ or $\tau(i,k) \geq \tau(j,k+1)$ then either $(i,k+1)\in \beta$ or both $(i,k+1) \in \alpha /\!\!/ \beta $ and $\tau(i,k+1) > \tau(j,k+1)$.
\end{enumerate}
\end{Definition}
A \textit{standard reverse composition tableau} (SRCT) $\ctau$ of shape $\alpha\vDash n$ is an SSRCT that contains every positive integer in $[n]$ exactly once.

Throughout this article, when considering an SSRCT, the boxes of the inner shape will be filled with white, while that of the outer shape will be filled with blue. The ordered pairs of positive integers $(i,j)$ such that $(i,j)\notin \alpha$ and $(i,j)\notin \beta$ are assumed to be filled with $0$s. This given, the slightly technical third condition in the definition of SSRCT is equivalent to the non-existence of the following triple configurations in the filling $\tau$
\begin{eqnarray*}
\ytableausetup{mathmode,boxsize=1.5em}
\begin{ytableau}
\cw & \myb y\\ \none & \none[\vdots]\\\none & \myb z
\end{ytableau}
\text{ and } z>y\geq 0,
\end{eqnarray*}
and
\begin{eqnarray*}
\ytableausetup{mathmode,boxsize=1.5em}
\begin{ytableau}
\myb x & \myb y\\ \none & \none[\vdots]\\\none & \myb z
\end{ytableau}
\text{ and }x \geq z>y\geq 0. \\
\end{eqnarray*}
The existence of a configuration of the above types in a filling will be termed a \textit{triple rule violation}. We will refer to the entry in the box in position $(i,j)$ of an SSRCT $\tau$ by $\tau_{(i,j)}$.
\begin{Example}\label{exampleskewssrctstraightsrct}
An SSRCT of skew reverse composition shape $(2,5,4,2)/\!\!/ (1,2,2)$ (left) and an SRCT of reverse composition shape $(3,4,2,3)$ (right).
\begin{eqnarray*}
\ytableausetup{mathmode,boxsize=1em}
\begin{ytableau}
\myb 4 & \myb 3\\ \cw&\myb 7 &\myb 7 &\myb 5 &\myb1\\\cw & \cw &\myb 6 &\myb 2\\ \cw & \cw
\end{ytableau}\hspace{15mm}
\begin{ytableau}
\myb 5& \myb 4&\myb 2\\\myb 9&\myb 7&\myb 6&\myb 3\\\myb 10&\myb 1\\\myb 12&\myb 11&\myb 8
\end{ytableau}
\end{eqnarray*}
\end{Example}

One more notion that we will use is that of the \textit{descent set} of an SRCT. Given an SRCT $\ctau$ of shape $\alpha\vDash n$, its descent set, denoted by $\des (\ctau)$, is defined to be the set of all integers $i$ such that $i+1$ lies weakly to the right of $i$ in $\ctau$. Note that $\des (\ctau)$ is clearly a subset of $[n-1]$. The \textit{descent composition} of $\ctau$, denoted by $\comp (\ctau)$, is the composition of $n$ associated with $\des (\ctau)$. As an example, the descent set of the SRCT in Example \ref{exampleskewssrctstraightsrct} is $\{1,2,5,7,9,10\}$. Hence the associated descent composition is $(1,1,3,2,2,1,2)$.

\subsection{Noncommutative symmetric functions}\label{NSym}
An algebra closely related to $\Lambda$ is the algebra of \textit{noncommutative symmetric functions} $\mathbf{NSym}$, introduced in \cite{GKLLRT}. The algebra $\mathbf{NSym}$ is the free associative algebra $\mathbb{Q}\langle \nch_1,\nch_2 ,\ldots \rangle$ generated by a countably infinite number of indeterminates $\nch_k$ for $k\geq 1$. Assigning degree $k$  to $\nch_k$, and extending this multiplicatively allows us to endow $\mathbf{NSym}$ with a structure of a graded algebra. A natural basis for the degree $n$ graded component of $\mathbf{NSym}$, denoted by $\mathbf{NSym}^{n}$, is given by the \textit{noncommutative complete homogeneous symmetric functions}, $\{\mathbf{h}_{\alpha}=\mathbf{h}_{\alpha_1}\cdots \mathbf{h}_{\alpha_k}:\alpha=(\alpha_1,\ldots,\alpha_k)\vDash n\}$.
The link between $\Lambda$ and $\mathbf{NSym}$ is made manifest through the \textit{forgetful} map, $\chi: \mathbf{NSym}\to \Lambda$, defined by mapping $\nch_i$ to $h_i$ and extending multiplicatively. Thus, the images of elements of $\mathbf{NSym}$ under $\chi$ are elements of $\Lambda$, imparting credibility to the term noncommutative symmetric function.

$\mathbf{NSym}$ has another important basis called the noncommutative ribbon Schur basis. We will denote the \textit{noncommutative ribbon Schur function} indexed by a composition $\beta$ by $\rib_{\beta}$. The following  \cite[Proposition 4.13]{GKLLRT} can be taken as the definition of the noncommutative ribbon Schur functions.
\begin{eqnarray*}
\rib_{\beta}&=&\displaystyle\sum_{\alpha \succcurlyeq \beta}(-1)^{l(\beta)-l(\alpha)}\nch_{\alpha}
\end{eqnarray*}
A multiplication rule for noncommutative ribbon Schur functions was proved in \cite{GKLLRT}, and we will be needing it later.
\begin{Theorem}\cite[Proposition 3.13]{GKLLRT}\label{ribbonmultiplication}
Let $\alpha=(\alpha_1,\ldots,\alpha_{k_1})$ and $\beta=(\beta_1,\ldots,\beta_{k_2})$ be compositions. Define two new compositions $\gamma$ and $\mu$ as follows
\begin{eqnarray*}
\gamma=(\alpha_1,\ldots,\alpha_{k_1},\beta_1,\ldots,\beta_{k_2}),\hspace{2mm} \mu=(\alpha_1,\ldots,\alpha_{k_1}+\beta_1,\beta_2,\ldots,\beta_{k_2}).
\end{eqnarray*}
Then 
\begin{eqnarray*}
\rib_{\alpha}\cdot \rib_{\beta}=\rib_{\gamma}+\rib_{\mu}.
\end{eqnarray*}
\end{Theorem}

We will also need the noncommutative analogue of the power sum symmetric functions. In \cite{GKLLRT}, they have actually defined two such analogues. Our interest is in the $\Psi$ basis. For a positive integer $n$, define
\begin{eqnarray}\label{def:noncommutativepowersum}
\Psi_n=\displaystyle\sum_{k=0}^{n-1}(-1)^k\rib_{(1^k,n-k)}.
\end{eqnarray}
This given, for a composition $\alpha=(\alpha_1,\ldots,\alpha_k)$, we define $\Psi_{\alpha}$ multiplicatively as $\Psi_{\alpha_1}\cdots\Psi_{\alpha_k}$. It is easy to show that $\chi(\Psi_n)=p_n$.

Now that we have defined the noncommutative analogue of power sum symmetric functions, to state a noncommutative Murnaghan-Nakayama rule, we need to define the analogue of the Schur functions in $\mathbf{NSym}$. The next subsection achieves that aim.
\subsection{Noncommutative Schur functions}
We will now describe a distinguished basis for $\mathbf{NSym}$, introduced in \cite{BLvW}, called the basis of \textit{noncommutative Schur functions}. They are naturally indexed by compositions, and the noncommutative Schur function indexed by a composition $\alpha$ will be denoted by $\ncs_{\alpha}$. They are defined implicitly using the relation
\begin{eqnarray}\label{def:implicitnoncommutativeschur}
\rib_{\beta}=\displaystyle\sum_{\alpha \vDash \vert\beta\vert} d_{\alpha,\beta}\ncs_{\alpha}
\end{eqnarray}
where $d_{\alpha,\beta}$ is the number of SRCTs of shape $\alpha$ and descent composition $\beta$.

The noncommutative Schur function $\ncsa$ satisfies the important property \cite[Equation 2.12]{BLvW}  that $$\chi(\ncsa)=s_{\widetilde{\alpha}}$$ where $\widetilde{\alpha}$ is the partition obtained by rearranging the parts of $\alpha$ in weakly decreasing order. Thus, the noncommutative Schur functions are lifts of the Schur functions in $\mathbf{NSym}$. They share many properties with the Schur functions. The interested reader should refer to \cite{BLvW, LMvW} for an in-depth study of these functions. Our main interest is in the noncommutative Pieri rules for the noncommutative Schur functions proved in \cite{BLvW} which we state below.

\begin{Theorem}\cite[Corollary 3.8]{BLvW}\label{thm: NoncommutativePieri}
Given a composition $\beta$, we have
\begin{eqnarray*}
\ncs_{(n)}\cdot \ncs_{\beta}=\sum_{\gamma}\ncsg
\end{eqnarray*} 
where the sum on the right runs over all compositions $\gamma >_{c} \beta$ such that $\vert\gamma/\!\!/\beta\vert = n$ and $\gamma/\!\!/\beta$ is a horizontal strip where the boxes have been added from left to right. Similarly,
\begin{eqnarray*}
\ncs_{(1^n)}\cdot \ncs_{\beta}=\sum_{\gamma}\ncsg
\end{eqnarray*} 
where the sum on the right runs over all compositions $\gamma >_{c} \beta$ such that $\vert\gamma/\!\!/\beta\vert = n$ and $\gamma/\!\!/\beta$ is a vertical strip where the boxes have been added from right to left.
\end{Theorem}

Next, we will establish an easy equality that will prove useful later.
\begin{Lemma}\label{ribbon-ncschur-equality-for-row-and-column}
Let $n\geq 0$. Then we have that $\ncs _{(n)}=\rib_{(n)}$ and $\ncs _{(1^n)} =\rib_{(1^n)}$.
\end{Lemma}
\begin{proof}
By \eqref{def:implicitnoncommutativeschur}, we have
\begin{eqnarray}
\rib _{\beta}=\sum_{\alpha \vDash n}d_{\alpha\beta}\ncsa .
\end{eqnarray}
Now, our claim follows once we note that if $\beta = (n)$ or $\beta = (1^n)$ then $d_{\alpha\beta}=\delta _{\alpha\beta}$. Here $\delta_{\alpha\beta}$ denotes the Kronecker delta function that equals $1$ if $\alpha=\beta$ and $0$ otherwise.
\end{proof}\label{eqRibbonSchur}

\subsection{Box-adding operators \texorpdfstring{$\upl_i$}{upleft}}
In this section, we will define box-adding operators that act on compositions and add a box to a composition in accordance with the Pieri rules stated in Theorem \ref{thm: NoncommutativePieri}.
\begin{Definition}
For every positive integer $i$, given a composition $\alpha$, $\upl_{i}(\alpha)$ is defined in the following manner.
$$
\upl_{i}(\alpha)= \left\lbrace\begin{array}{lll} \beta &  \text{ if $\alpha \lessdot_{c} \beta$ and $\beta/\!\! /\alpha$ is a box }\\\text{} & \text{ in the $i$-th column}\\ 0 & \text{otherwise.}  \end{array}\right.
$$
\end{Definition}
It is clear from the above definition that $\upl_i(\alpha)$ is non-zero if and only if $\alpha$ has a part equal to $i-1$.
\begin{Example}
Let $\alpha = (2,1,3,1)$. Then $\upl_{1}(\alpha)=(1,2,1,3,1)$, $\upl_{2}(\alpha)=(2,2,3,1)$ and $\upl_{5}(\alpha)=0$.
\end{Example}

It can be easily seen that the box-adding operators satisfy the following relation.
\begin{Lemma}\label{modifiedlocalplactic}
Given positive integers $i,j$ such that $\vert i-j\vert\geq 2$, we have the following equality as operators on compositions.
\begin{eqnarray*}
\upl_i\upl_j&=&\upl_j\upl_i
\end{eqnarray*}
\end{Lemma}
\begin{proof}
We will assume, without loss of generality that $i<j$. Consider first the case where $i=1$ and $j\geq 3$.  If $\alpha = (\alpha_1,\ldots,\alpha_k)$ is a composition, and no part of $\alpha$ equals $j{-}1$, then we have $\upl_1\upl_j (\alpha)=\upl_j\upl_1(\alpha)=0$. If, on the other hand, there is a part of $\alpha$ that equals $j{-}1$, pick the leftmost such part. Let this part be $\alpha_r$. Then $\upl_1\upl_j(\alpha)=(1,\alpha_1,\ldots,\alpha_r{+}1,\ldots,\alpha_k)=\upl_j\upl_1(\alpha)$.

Now consider the case where $i\geq 2$. Clearly, if either $i{-}1$ or $j{-}1$ is not a part of $\alpha$, we have $\upl_i\upl_j(\alpha)=\upl_j\upl_i(\alpha)=0$. If both $i{-}1$ and $j{-}1$ are parts of $\alpha$,  let $\alpha_r$ and $\alpha_s$ be their leftmost occurrences in $\alpha$ respectively. It is easily seen that the application of either $\upl_i$ first or $\upl_j$ first doesn't change the position of the leftmost instance of a part equalling $j{-}1$ or $i{-}1$ in $\alpha$ respectively. Hence $\upl_i\upl_j=\upl_j\upl_i$ for $|i{-}j|\geq 2$.
\end{proof}

\subsection{SRCTs associated with sequence of box-adding operators}\label{maximalchainbijection}
Let $\alpha$ be a composition and suppose $w=\upl_{i_1}\cdots \upl_{i_n}$ is such that $w(\alpha)=\beta$, where $\beta$ is not $0$. Then, by \cite[Proposition 2.11]{BLvW}, we see that there is a unique SRCT of shape $\beta/\!\!/\alpha$ associated with the word $w$, as we can associate the following maximal chain in $\mathcal{L}_{c}$ with $w$
\begin{eqnarray*}
\alpha \lessdot_{c} \upl_{i_n}(\alpha)\lessdot_{c} \cdots \lessdot_{c} \upl_{i_2}\cdots\upl_{i_n}(\alpha)\lessdot_{c}\upl_{i_1}\cdots\upl_{i_n}(\alpha)=\beta.
\end{eqnarray*}
This maximal chain in turn gives rise to a unique SRCT of size $n$ and shape $\beta/\!\!/\alpha$, where $n-i+1$ is placed in the $i$-th box added.
\begin{Example}
Let $\alpha=(1,3,2)$ and $\beta=\upl_{2}\upl_{4}\upl_{1}\upl_{2}\upl_3(\alpha)=(2,2,4,3)$. Then $w=\upl_2\upl_{4}\upl_{1}\upl_{2}\upl_3$ uniquely corresponds to the following SSRCT of shape $\beta/\!\!/\alpha$.
\begin{eqnarray*}
\ytableausetup{mathmode,boxsize=1em}
\begin{ytableau}
\myb 3 & \myb 1\\\cw & \myb 4\\ \cw &\cw &\cw & \myb 2\\ \cw &\cw & \myb 5
\end{ytableau}
\end{eqnarray*}
\end{Example}

\subsection{Pieri rules using box-adding operators}
Let $\mathbb{C}Comp$ denote the vector space consisting of formal sums of compositions. Consider the map $\Phi : \mathbf{NSym} \longrightarrow \mathbb{C}Comp$ defined by sending
$\ncs_{\alpha} \mapsto \alpha$
and extending linearly. We give $\mathbb{C}Comp$ an algebra structure by defining the product between two composition $\alpha\cdot \beta$ as follows.
\begin{eqnarray*}
\alpha\cdot \beta&=&\Phi(\Phi^{-1}(\alpha)\cdot \Phi^{-1}(\beta))
\end{eqnarray*}
It is easily seen that with the above definition of $\Phi$, we have an $\mathbb{C}$-algebra isomorphism between $\mathbf{NSym}$ and $\mathbb{C}Comp$. We will freely use the notation $\alpha\cdot\beta$ to denote the product $\ncsa\cdot \ncsb$, and for the sake of convenience, we will not distinguish between $\ncsa\cdot \ncsb$ and $\Phi(\ncsa\cdot \ncsb)$.

With the above setup and armed with the definition of box-adding operators, we restate Theorem \ref{thm: NoncommutativePieri} in the language of box-adding operators.
\begin{Proposition}\label{lpr}
Let $\alpha$ be a composition and $n$ be a positive integer. Then
\begin{eqnarray*}
(n)\cdot\alpha&=& \left(\displaystyle\sum_{i_n>\cdots>i_1}\upl_{i_n}\cdots\upl_{i_1}\right)(\alpha),\\
(1^n)\cdot\alpha&=&\left(\displaystyle\sum_{i_n\leq \cdots\leq i_1}\upl_{i_n}\cdots\upl_{i_1}\right)(\alpha).
\end{eqnarray*}
\end{Proposition}
We would like to compute $\rib_{(1^k,n-k)}\cdot \ncsa$ using box-adding operators. Note that we already know the answer in the special cases $k=0$ or $k=n-1$, as these are precisely the cases encompassed by Proposition \ref{lpr}. To this end, we will need some more notation, which is precisely what the aim of the next subsection is.

\subsection{Reverse hookwords and multiplication by \texorpdfstring{$\rib_{(1^k,n-k)}$}{rib}}
Given positive integer $n$ and $k$ such that $0\leq k\leq n{-}1$, define $w=\upl_{i_1}\ldots \upl_{i_n}$ to be a \textit{reverse $k$-hookword} if $i_1\leq i_2\leq \cdots \leq i_{k+1} >i_{k+2}>\cdots > i_{n}$. If $w$ is a reverse $k$-hookword for some $k$, then we will call $w$ a \textit{reverse hookword}. 
Denote by $\supp (w)$ the set formed by the indices $i_1,\ldots, i_n$ (by discarding duplicates). Call $w$ \textit{connected} if $\supp(w)$ is an interval in $\mathbb{Z}_{\geq 1}$, otherwise call it \textit{disconnected}. Put differently, $w$ is connected if the set of indices, when considered in ascending/descending order, form a contiguous block of positive integers. Let
\begin{eqnarray*}
&RHW_n = \text{set of reverse hookwords of length }n,\\
&CRHW_n = \text{set of connected reverse hookwords of length }n.
\end{eqnarray*}
Let $RHW_{n,k}$ denote that subset of $RHW_n$ consisting of reverse $k$-hookwords. 

For the same $w$ as before, we define the \emph{content vector} of $w$ to be a finite ordered list of non-negative integers $\alpha= (\alpha_1,\alpha_2,\ldots)$ where $\alpha_{i}$ denotes the number of times the operator $\upl_{i}$ appears in $w$.  Also associated with $w$ are the notions of $\arm(w)$ and $\leg(w)$ defined as follows.
\begin{eqnarray*}
&\arm(w)&=\{i_j: 1\leq j\leq k+1\}\nonumber\\&\leg(w)&=\{i_j: k+1\leq j\leq n\}
\end{eqnarray*}
Finally, define $\asc(w)$ to be  $\vert \arm(w)\vert-1$.

\begin{Example}
Let $w=\upl_2\upl_5\upl_6\upl_7\upl_8\upl_8\upl_9\upl_9\upl_7\upl_6\upl_4\upl_3$. Then $w$ is a connected reverse $7$-hookword. The content vector of $w$ is given by $(0,1,1,1,1,2,2,2,2)$. Furthermore, we also have that
\begin{eqnarray*}
&\arm(w)&=\{9,9,8,8,7,6,5,2\},\\
&\leg(w)&=\{3,4,6,7,9\}.
\end{eqnarray*}
\end{Example}

Now, using reverse hookwords, we can express the multiplication of noncommutative Schur functions by noncommutative ribbon Schur functions in terms of box-adding operators. 
\begin{Lemma}\label{actionribbon}
Let $\beta=(1^k,n-k)$ be a composition where $0\leq k\leq n-1$, and $\alpha$ be a composition. Then
\begin{eqnarray*}\rib_\beta\cdot \ncsa=\left( \displaystyle\sum_{w\in RHW_{n,k}} w\right) (\alpha).\end{eqnarray*}
\end{Lemma}
\begin{proof}Note first that in the case $k=0$, the claim is true as it is equivalent to Proposition \ref{lpr}. We will establish the claim by induction on $k$. Assume that the claims holds for all integers $\leq k$ for some $k\geq 1$. We will compute $(\rib_{(1^{k+1})}\cdot\rib_{(n-k-1)})\cdot \ncsa$ in two ways. First notice that by using Theorem \ref{ribbonmultiplication}, we get
\begin{eqnarray}\label{ehfirstway}
 (\rib_{(1^{k+1})}\cdot\rib_{(n-k-1)})\cdot \ncsa &=& \rib_{(1^{k+1},n-k-1)}\cdot \ncsa + \rib_{(1^k,n-k)}\cdot \ncsa .
\end{eqnarray}
Now, using Proposition \ref{lpr}, we have that 
\begin{eqnarray}\label{ehsecondway}
(\rib_{(1^{k+1})}\cdot\rib_{(n-k-1)})\cdot \ncsa&=& \left(\displaystyle\sum_{i_1\leq \cdots\leq i_{k+1}} \upl_{i_1}\cdots\upl_{i_{k+1}}\right)\left(\displaystyle\sum_{i_{k+2}>\cdots >i_n}\upl_{i_{k+2}}\cdots\upl_{i_n}\right)(\alpha)\nonumber\\&=& \left( \displaystyle\sum_{w\in RHW_{n,k}} w\right) (\alpha)+\left( \displaystyle\sum_{w'\in RHW_{n,k+1}} w'\right) (\alpha).
\end{eqnarray}
By induction hypothesis we know that 
\begin{eqnarray}\label{ehcombined}
\rib_{(1^k,n-k)}\cdot \ncsa &=&\left( \displaystyle\sum_{w\in RHW_{n,k}} w\right) (\alpha).
\end{eqnarray}
Now, using \eqref{ehfirstway}, \eqref{ehsecondway} and \eqref{ehcombined}, the claim follows.
\end{proof}

\section{Multiplication by noncommutative power sums in terms of box-adding operators}\label{sec:boxaddingmurnaghannakayama}
 On using the definition of the noncommutative power sum function $\Psi_n$ in terms of the noncommutative ribbon Schur functions and Lemma \ref{actionribbon} subsequently, we get that \begin{eqnarray}\label{primordialMurNak}
\Psi_n \cdot \ncsa &=& \left(\sum_{k=0}^{n-1}(-1)^k\rib_{(1^k,n-k)}\right)\cdot \ncsa\nonumber \\ &=&\sum_{k=0}^{n-1}(-1)^k\sum_{w \in RHW_{n,k}}w (\alpha)\nonumber \\ &=&\sum_{w\in RHW_n} (-1)^{\asc(w)}w (\alpha).
\end{eqnarray}
Using the involution described in Section $5$ of \cite{fomin-greene-1}, we can restate \eqref{primordialMurNak} in the following form.
\begin{eqnarray}\label{advMurNak}
\Psi_n \cdot \ncsa &=& \sum_{w \in CRHW_n} (-1)^{\asc(w)}w (\alpha)
\end{eqnarray}
While the equation above is a legitimate way to compute $\Psi_n\cdot \ncsa$, it is not cancellation-free. Achieving this will be our aim in the next section. 
\begin{Example}
As an example, we will compute $\Psi_2\cdot \ncsa$ where $\alpha=(1,3,2)$. The possible elements of $CRHW_2$ are $\upl_i\upl_i$, $\upl_i\upl_{i+1}$ and $\upl_{i+1}\upl_i$ for $i\geq 1$. Thus, we have
\begin{eqnarray*}
\Psi_2\cdot \ncs_{(1,3,2)} &=& \sum_{i\geq 1}(\upl_{i+1}\upl_i{-}\upl_i\upl_{i+1}{-}\upl_i^2)((1,3,2))\\ &=& (\upl_2\upl_1{+}\upl_3\upl_2{+}\upl_4\upl_3{+}\upl_5\upl_4{-}\upl_1\upl_2{-}\upl_2\upl_3-\upl_3\upl_4-\upl_1\upl_1)((1,3,2))\\&=&(2,1,3,2){+}(3,3,2){+}(1,4,3){+}(1,5,2){-}(1,2,3,2)\\&&{-}(2,3,3){-}(1,4,3){-}(1,1,1,3,2)\\&=&(2,1,3,2){+}(3,3,2){+}(1,5,2){-}(1,2,3,2){-}(2,3,3){-}(1,1,1,3,2).
\end{eqnarray*}
Note that this example reaffirms our claim that \eqref{advMurNak} is not cancellation-free. The composition (1,4,3) appeared once with coefficient 1, and once with coefficient -1. Hence it does not appear in the final result.
\end{Example}

\section{The Murnaghan-Nakayama rule for noncommutative Schur functions}\label{sec:finalmurnaghannakayama}
By \eqref{advMurNak} we know that 
\beqn 
\Psi_n\cdot \ncsa&=&\displaystyle \sum_{\beta}k_{\beta}\ncsb
\eeqn
where $$k_{\beta}=\displaystyle\sum_{\substack{w\in CRHW_n\\w(\alpha)=\beta}}(-1)^{\asc(w)}.$$
Thus, our aim is to compute $k_{\beta}$ given any composition $\beta$. Note that if we are given a $\beta$ such that there exists a $w\in CRHW_n$ with $w(\alpha)=\beta$, then $\alpha <_{c}\beta$ and the skew reverse composition shape $\beta/\!\!/\alpha$ has $n$ boxes, i.e. $\lvert \beta/\!\!/\alpha\rvert =n$.

Suppose $\alpha<_{c}\beta$. Define the \textit{support}, $\supp(\beta/\!\!/\alpha)$, as follows.
\begin{eqnarray*}
\supp(\beta/\!\!/ \alpha)=\{j:(i,j)\in \beta/\!\!/\alpha\}.
\end{eqnarray*}
The skew reverse composition shape $\beta/\!\!/\alpha$ whose support is an interval in $\mathbb{Z}_{\geq 1}$ will be called an \textit{interval shape}. Finally, an interval shape $\beta/\!\!/\alpha$  will be called an \textit{nc border strip} if it satisfies the following conditions.
\begin{enumerate}
\item Suppose there exist boxes in positions $(i,1)$ and $(i,2)$ in $\beta/\!\!/\alpha$. Then the box in position $(i,1)$ is the bottommost box in column $1$ of $\beta/\!\!/\alpha$.
\item Suppose there exist boxes in positions $(i,j)$ and $(i,j+1)$ in $\beta/\!\!/\alpha$ for $j\geq 2$. Then the box in position $(i,j)$ is the topmost box in column $j$ of $\beta/\!\!/\alpha$.
\end{enumerate}
Given an nc border strip $\beta/\!\!/\alpha$, its \textit{height} $ht(\beta/\!\!/\alpha)$ is defined to be one less than the number of row it occupies.
 
Suppose we are given an interval shape $\beta/\!\!/\alpha$. We discuss three notions that have to do with the relative positions of boxes in consecutive columns of $\beta/\!\!/\alpha$. We say that column $j+1$ is \textit{south-east} of column $j$ in $\beta/\!\!/\alpha$ if there is a box in column $j+1$ that is strictly south-east of a box in column $j$. 
We say that column $j+1$ is \textit{north-east} of column $j$, if for any pair of  boxes $(i_1,j+1)$ and $(i_2,j)$ in $\beta/\!\!/\alpha$, we have $i_1<i_2$. Here, we assume that there exists at least one box in the $j+1$-th column.
Finally, we say that column $j+1$ is \textit{east} of column $j$ if there are boxes in positions $(i,j)$ and $(i,j+1)$ for some $i$.

This given, we associate three statistics with an interval shape $\beta/\!\!/\alpha$ as follows. 
\begin{eqnarray*}
E(\beta/\!\!/\alpha)&=&\{j: \text{column $j+1$ is east of column $j$}\}\\
SE(\beta/\!\!/\alpha)&=&\{j:\text{column $j+1$ is south-east of  column $j$ and $j\notin E(\beta/\!\!/\alpha)$ }\}\\
NE(\beta/\!\!/\alpha)&=&\{j:\text{column $j+1$ is north-east of column $j$}\}
\end{eqnarray*}
\begin{Example}\label{exampleforelementofbalpha}
Consider the following interval shape $\beta/\!\!/\alpha$, where the unfilled boxes belong to the inner shape and those shaded blue belong to the outer shape.
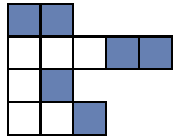
\begin{eqnarray*}
\ytableausetup{centertableaux}
\begin{ytableau}
\myb & \myb\\ \cw& \cw &\cw &\myb&\myb\\\cw & \myb\\ \cw & \cw &\myb
\end{ytableau}
\end{eqnarray*}
In this case we have $E(\beta/\!\!/\alpha)=\{1,4\},
SE(\beta/\!\!/\alpha)=\{2\},
NE(\beta/\!\!/\alpha)=\{3\}.$ Notice that the above interval shape is actually an nc border strip whose height is $3$.
\end{Example}

\begin{Remark}
Given an interval shape $\beta/\!\!/\alpha $, note that $E(\beta/\!\!/\alpha)$, $SE(\beta/\!\!/\alpha)$, $NE(\beta/\!\!/\alpha)$ are always pairwise disjoint. Furthermore, if $p$ is the maximum element of $\supp(\beta/\!\!/\alpha)$, then we have 
\begin{eqnarray*}
\supp(\beta/\!\!/\alpha)&=&\{p\}\uplus E(\beta/\!\!/\alpha)\uplus SE(\beta/\!\!/\alpha) \uplus NE(\beta/\!\!/\alpha).
\end{eqnarray*}
\end{Remark}
Now, consider the following two sets 
\begin{eqnarray*}
A_{\alpha,n}&=&\{\beta:\beta = w(\alpha) \text{ for some } w\in CRHW_n\},\\
B_{\alpha,n}&=&\{\beta: \alpha <_{c}\beta\text{ and } \beta/\!\!/\alpha \text{ is an nc border strip of size } n\}.
\end{eqnarray*}
Our aim is to establish that $A_{\alpha,n}=B_{\alpha,n}$. This means that we will be able to replace checking whether there exists $w\in CRHW_n$ such that $w(\alpha)=\beta$ with checking whether $\beta/\!\!/\alpha$ is an nc border strip.

\begin{Lemma}\label{reversehookwordgivenncborderstrip}%this lemma gives a reverse hookword given a nc border strip
$$B_{\alpha,n}\subseteq A_{\alpha,n}.$$
\end{Lemma}
\begin{proof}
To establish the claim, we need to show that given $\beta\in B_{\alpha,n}$ (i.e. $\beta/\!\!/\alpha$ is an nc border strip), there exists a $w\in CRHW_n$ such that $w(\alpha)=\beta$.
We will construct an SRCT of shape $\beta/\!\!/\alpha$ that
will be associated (in the sense of Section \ref{maximalchainbijection}) with an element of $CRHW_n$.

Let
$E(\beta/\!\!/\alpha)=\{x_1<\cdots <x_s\}$. For every $x_i$, place the integer $n-i+1$ in the topmost box of the $x_i$-th
column of $\beta/\!\!/\alpha$ if $x_i>1$ and in the bottommost box if $x_i=1$. After this step, traverse the remaining $n-s$ unfilled boxes in $\beta/\!\!/\alpha$ in the manner described below and place the integers $n-s$ down to $1$ in that order.
\begin{itemize}
\item Start from the rightmost column that contains an empty box.
\item Visit every empty box from top to bottom if column under consideration is not the first column, otherwise visit boxes from bottom to top.
\item Repeat the above steps for the next rightmost column in  $\beta/\!\!/\alpha$ that contains an empty box.
\end{itemize} 
We claim that the resulting filling, which we call $\tau$, is an SRCT. By construction, if
$1\in \supp(\beta/\!\!/\alpha)$, then the entries in the boxes in
column $1$ strictly increase from top to bottom. In all other columns,
the entries decrease from top to bottom.

To show that entries decrease
from left to right along rows, assume first the existence of boxes in
position $(i,j)$
and $(i,j+1)$ in $\beta/\!\!/\alpha$. Clearly $j\in
E(\beta/\!\!/\alpha)$. Since $\beta/\!\!/\alpha$ is an nc border strip, we know that the box in position
$(i,j)$ is the topmost box in the $j$-th column  of $\beta/\!\!/\alpha$ if $j\geq 2$, and is the bottommost box if $j=1$. By our construction of the filling, if $j\in E(\beta/\!\!/\alpha)$, every box of $\beta/\!\!/\alpha$ in column $j+1$ contains a number strictly less than $\ctau_{(i,j)}$. Hence in particular, $\tau_{(i,j)}>\tau_{(i,j+1)}$ and thus, the entries decrease from left to right along rows.

Now we have to show that there are no triple rule violations in $\tau$. Since, in every column $k$ where $k\geq 2$, the entries decrease from top to bottom, we are guaranteed there are no triple rule violations of the form 
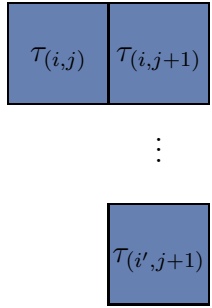
\begin{eqnarray*}
\ytableausetup{mathmode, boxsize=3.2em}
\begin{ytableau}\myb
  \tau_{(i,j)}& \myb\tau_{(i,j+1)}\\\none &\none[\vdots]\\\none &\myb\tau_{(i',j+1)}\end{ytableau}
\end{eqnarray*}
where $j\geq 1$, and $(i',j+1)$, $(i,j+1)\in \beta/\!\!/\alpha$ with $i'>i$.

Now assume that $(i,j+1)\notin \beta/\!\!/\alpha$. We will show that $\ctau_{(i,j)}<\ctau_{(i',j+1)}$ for any $i'>i$ where $(i',j+1)\in \beta/\!\!/\alpha$.
Assume first that $j\in E(\beta/\!\!/\alpha)$. If $j=1$, then we know that $(i,j)$ is not the bottommost box in column $j$ of $\beta/\!\!/\alpha$, as $(i,j+1)\notin \beta/\!\!/\alpha$. Thus, by construction, $\tau_{(i,j)}$ is strictly less than every entry in column $2$. In particular, $\tau_{(i',j+1)}>\tau_{(i,j)}$, and we do not have triple rule violations. If $j\geq 2$, then we know that the box in position $(i,j)$ is not the topmost box in column $j$. Again, we have that $\tau_{(i,j)}$ is strictly less than every entry in column $j+1$, and as before, we have no triple rule violations.

Now suppose that $(i,j+1)\notin \beta/\!\!/\alpha$, and that $j\notin E(\beta/\!\!/\alpha)$. Then, the way the filling $\tau$ has been constructed implies that the greatest entry in column $j$ is less than the smallest entry in column $j+1$. This guarantees that $\tau_{(i',j+1)}>\tau_{(i,j)}$, and we have no triple rule violations.

Hence, $\tau$ is an SRCT. Clearly, the word associated with $\tau$ is an element of $RHW_n$. The fact that it is actually an element of $CRHW_n$ follows from the fact that $\beta/\!\!/\alpha$ is an interval shape. Hence there exists a $w\in CRHW_n$ such that $w(\alpha)=\beta$, implying $\beta \in A_{\alpha,n}$.
\end{proof}

Next, we give an example of the construction presented in the algorithm above.
\begin{Example}
Consider the shape $\beta/\!\!/\alpha$ presented in Example \ref{exampleforelementofbalpha}.
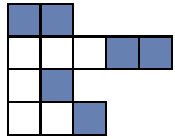
\begin{eqnarray*}
\ytableausetup{centertableaux,boxsize=1em}
\begin{ytableau}
\myb & \myb\\ \cw& \cw &\cw &\myb&\myb\\\cw & \myb\\ \cw & \cw &\myb
\end{ytableau}
\end{eqnarray*}
Here $\vert \beta/\!\!/\alpha\vert=6$. Hence we will be placing the numbers $1$ to $6$ in the blue boxes above. Since $E(\beta/\!\!/\alpha)=\{1,4\}$, the bottommost box in column $1$ of $\beta/\!\!/\alpha$ will have $6$ placed in it, while the topmost box in column $4$ will have $5$ placed in it. Now the columns with empty boxes, considered from right to left, are $5$, $3$ and $2$. We fill in these boxes with the numbers $4,3,2,1$ in that order from top to bottom in each column to obtain the following filling.
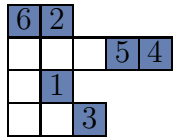
\begin{eqnarray*}
\ytableausetup{centertableaux,boxsize=1em}
\begin{ytableau}
\myb6 & \myb2\\ \cw& \cw &\cw &\myb 5&\myb4\\\cw & \myb1\\ \cw & \cw &\myb3
\end{ytableau}
\end{eqnarray*}
\end{Example}
\begin{Lemma}\label{orderofentriesincolumns}
Suppose $w\in CRHW_n$ is such that $w(\alpha)=\beta$. Let $\tau$ be the SRCT of shape $\beta/\!\!/\alpha$ corresponding to the word $w$. Then, the entries in $\tau$ in column $j$ strictly decrease from top to bottom for all $j\geq 2$. 
\end{Lemma}
\begin{proof}
Suppose $w$ has $c$ occurrences of $\upl_{j}$ where $j\geq 2$. We can then factorize $w$ as follows
\begin{eqnarray}
w=w_3\upl_{j}^{c-1}w_2\upl_{j}w_1,
\end{eqnarray}
where $w_1,w_2$ could be empty words. Notice that since $w$ is a reverse hookword, all the operators that constitute $w_2$ are of the form $\upl_k$ where $k>j$. Hence they only add boxes to parts of length $\geq j$ in the composition $\upl_{j}w_1(\alpha)$.
%%Thus, if we restrict our attention to only the first $j$ columns of $w_2\upl_jw_1(\alpha)$, then the resulting shape is $\upl_jw_1(\alpha)$. 
Now $\upl_j^{c-1}$ adds boxes in the $j$-th column. %%of $\upl_jw_1(\alpha)$. 
%%The resulting shape is $\upl_j^{c-1}(\upl_jw_1(\alpha))=\upl_j^{c}(w_1(\alpha))$.%%
Thus, we have added $c$ boxes in the $j$-th column and the definition of the box-adding operators implies that repeated applications of $\upl_j$ add boxes from top to bottom, when $j\geq 2$. Hence the entries in these $c$ boxes strictly decrease from top to bottom.
\end{proof}

Next, we note down a couple of lemmas that are immediate consequences of the bijection mentioned in Subsection \ref{maximalchainbijection}. For Lemmas \ref{observation1} and \ref{observation2}, and Corollary \ref{notgreatersmallerthaneverythingelse}, we will work under the assumption that $w\in CRHW_n$ is such that $w(\alpha)=\beta$, and that $\tau$ is the SRCT of shape $\beta/\!\!/\alpha$ corresponding to $w$. We will assume further that $j$ belongs to $\supp (\beta/\!\!/\alpha)$ (or $\supp(w)$), but is not the greatest element therein.
\begin{Lemma}\label{observation1}
Suppose $j\notin \leg(w)$ for some $j\geq 1$, then the greatest entry in column $j$ of $\tau$ is  strictly less than the smallest entry in column $j+1$ in $\tau$.
\end{Lemma}
\begin{Lemma}\label{observation2}
Suppose $j\in \leg(w)$ for some $j\geq 1$.  Then
\begin{itemize}
\item the greatest entry  in column $j$ of $\tau$ is strictly greater than than the greatest entry in column $j+1$ in $\tau$, and
\item all other entries in column $j$ of $\tau$ are strictly smaller than the smallest entry in column $j+1$ in $\tau$.
\end{itemize}
\end{Lemma}
We will note down an important corollary of the two lemmas above.
\begin{Corollary}\label{notgreatersmallerthaneverythingelse}
For all $j\geq 1$, every entry in column $j$ except the greatest one, is strictly smaller than the smallest entry in column $j+1$.
\end{Corollary}
\begin{Lemma}
\begin{eqnarray*}
A_{\alpha,n}=B_{\alpha,n}.
\end{eqnarray*}
\end{Lemma}
\begin{proof}
We have already shown in Lemma \ref{reversehookwordgivenncborderstrip} that every element of $B_{\alpha,n}$ is an element of $A_{\alpha,n}$. Thus, now we need to show that $A_{\alpha,n}\subseteq B_{\alpha,n}$.

Assume $\beta\in A_{\alpha,n}$. We have to show that $\beta\in B_{\alpha,n}$ as well. It is clear that if $\beta=w(\alpha)$ for some $w\in CRHW_n$, then $\alpha <_{c}\beta$.   The fact that  $\beta/\!\!/\alpha$ is an interval shape of size $n$ follows from the fact that $w$ is connected and has length $n$. Let $\tau$ denote the SRCT of shape $\beta/\!\!/\alpha$ corresponding to $w$.

Next, suppose there is a configuration of the following type. 
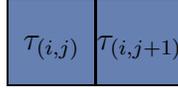
\begin{eqnarray*}
\ytableausetup{mathmode,boxsize=2.8em}
\begin{ytableau}
\myb\tau_{(i,j)}& \myb\tau_{(i,j+1)}
\end{ytableau}
\end{eqnarray*}
If $\tau_{(i,j)}$ is not the greatest entry in column $j$, then Corollary \ref{notgreatersmallerthaneverythingelse} implies that $\tau_{(i,j)}<\tau_{(i,j+1)}$, which contradicts the fact that $\tau$ is an SRCT. Hence $\tau_{(i,j)}$ is the greatest entry in column $j$. Hence the box in position $(i,j)$ is the bottommost box in column $j$ if $j=1$, or the topmost box in column $j$ otherwise, by invoking Lemma \ref{orderofentriesincolumns}.

This finishes the proof that $\beta/\!\!/\alpha$ is an nc border strip if $\beta\in A_{\alpha,n}$. Thus $A_{\alpha,n}\subseteq B_{\alpha,n}$, and we are done.
\end{proof}

We will outline our strategy for the remainder of this section. Now that we have established that $A_{\alpha,n}=B_{\alpha,n}$, our next aim is to enumerate all hookwords $w$ such that $w(\alpha)=\beta$ where $\beta\in B_{\alpha,n}$, i.e. $\beta/\!\!/\alpha$ is an nc border strip of size $n$. This will be achieved in Lemma \ref{cardinalityvalidhookwords}. Once this is done, we will be able to compute the coefficient $$k_{\beta}=\displaystyle\sum_{\substack{w\in CRHW_n\\w(\alpha)=\beta}}(-1)^{\asc (w)}.$$ Recall that $k_{\beta}$ was defined to be the coefficient of $\ncsb$ in the expansion $\Psi_n\cdot \ncsa$. This will give us a noncommutative analogue of the Murnaghan-Nakayama rule that we will state in Theorem \ref{MurNakfinal}.
\begin{Lemma}\label{marriedimpliesleg}
Given $\beta\in B_{\alpha,n}$, if for any $w\in CRHW_n$ such that $w(\alpha)=\beta$ we have that column $j\in E(\beta/\!\!/\alpha)$, then $j\in \leg(w)$.
\end{Lemma}
\begin{proof}
Let $\tau$ be the SRCT of shape $\beta/\!\!/\alpha$ corresponding to the word $w=\upl_{i_1}\cdots\upl_{i_n}$.
If $j\in E(\beta/\!\!/\alpha)$, then there exists a configuration in $\beta/\!\!/\alpha$ of the following form.
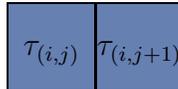
\begin{eqnarray*}
\ytableausetup{mathmode,boxsize=2.8em}
\begin{ytableau}
\myb\tau_{(i,j)}& \myb\tau_{(i,j+1)}
\end{ytableau}
\end{eqnarray*}
We know that $\tau_{(i,j)} > \tau_{(i,j+1)}$, and this implies that there exists $1\leq p<q\leq n$ such that $\upl_{i_p}=\upl_{j+1}$ and $\upl_{i_q}=\upl_{j}$.  But since $w$ is a reverse hookword, this immediately implies that $j\in \leg(w)$.
\end{proof}

\begin{Lemma}\label{dominationimpliesnotinleg}
Given $\beta\in B_{\alpha,n}$, if for any $w\in CRHW_n$ such that $w(\alpha)=\beta$  we have that column $j\in SE(\beta/\!\!/\alpha)$, then $j\notin \leg(w)$.
\end{Lemma}
\begin{proof}
Let $\tau$ be the SRCT of shape $\beta/\!\!/\alpha$ corresponding to the word $w$.
Assume first that $j=1$. Let the bottommost box in the first column of $\beta/\!\!/\alpha$ be in position $(i,1)$. Since $1\in SE(\beta/\!\!/\alpha)$, we know that all boxes in column $2$ that belong to $\beta/\!\!/\alpha$ lie strictly southeast of the box in position $(i,1)$. Notice that $\tau_{(i,1)}$ is the greatest amongst all the numbers in column $1$ of $\tau$. We have the following configuration in $\tau$.
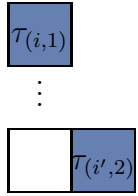
\begin{eqnarray*}
\ytableausetup{mathmode,boxsize=2em}
\begin{ytableau}
\myb \tau_{(i,1)}\\\none[\vdots]\\ \cw & \myb\tau_{(i',2)}
\end{ytableau}
\end{eqnarray*}
Since $\tau$ is an SRCT, we know that $\tau_{(i',2)}>\tau_{(i,1)}$. Thus in particular, the smallest entry in the second column of $\tau$ is strictly greater than $\tau_{(i,1)}$. Thus, the first time a box is added in the first column happens after all the boxes that were to be added in the second column by $w$ have been added. Hence $1\notin \leg(w)$.

Now assume $j\geq 2$. Let $\tau_{(i,j)}$ be the entry in the topmost box of the $j$-th column of $\tau$. This entry is greater than every other entry in the $j$-th column of $\tau$ by Lemma \ref{orderofentriesincolumns}. Since $j\in SE(\beta/\!\!/\alpha)$, we are guaranteed the existence of a box in position $(i',j+1)$ where $i'>i$. Arguing like before, we must have $\tau_{(i',j+1)}>\tau_{(i,j)}$. Hence there is at least one entry in the $j+1$-th column of $\tau$ that is greater than every entry in the $j$-th column of $\tau$. Given that $w$  is a reverse hookword, this implies that $j\notin \leg(w)$.
\end{proof}

\begin{Lemma}\label{submissionjgeq2}
Let $\beta\in B_{\alpha,n}$. If $j\in NE(\beta/\!\!/\alpha)$, then $j\geq 2$.
\end{Lemma}
\begin{proof}
Let $w\in CRHW_n$ be such that $w(\alpha)=\beta$.
Assume that the claim is not true. Hence $1\in NE(\beta/\!\!/\alpha)$. Thus, we must have that $\{1,2\}\subseteq \supp(\beta/\!\!/\alpha)$. If $1\in \leg(w)$, then the way the box-adding operators act implies that $1\in E(\beta/\!\!/\alpha)$, as $\upl_2$ will add a box adjacent to a newly added box resulting from the operator $\upl_1$ applied before it.

Now, assume $1\notin \leg(w)$. Then $w$ factors uniquely as $\upl_{1}^kw_1$, where $w_1$ has no instances of $\upl_1$. Then we know that in computing $\upl_{1}^k(w_1(\alpha))$, the boxes added in column $1$ will be strictly northwest of the boxes added in column $2$ while computing $w_1(\alpha)$. This also follows from how $\upl_1$ acts. But then $1\in SE(\beta/\!\!/\alpha)$. Again, this is a contradiction. 
Hence we must have that $j\geq 2$.
\end{proof}

\begin{Lemma}\label{commutativityconsecutive}
Let $j\geq 2$ be a positive integer and $\mu$ be a composition. Suppose that $\mu$ has parts equalling $j$ and $j-1$ and that the number of parts that equal $j$ and lie to the left of the leftmost instance of a part equalling $j-1$ is $m$. Then for all $0\leq k\leq m$, we have 
\begin{eqnarray*}
\upl_j\upl_{j+1}^{k}(\mu)=\upl_{j+1}^{k}\upl_j(\mu).
\end{eqnarray*}
\end{Lemma}
\begin{proof}
The case where $m=0$ is trivial. Assume that $m\geq 1$.

The operator $\upl_{j+1}^{k}$ adds a box to each of the $k$ leftmost parts in $\mu$ that equal $j$. We know that there are $m$ parts in $\mu$ that equal $j$ and lie to the left of the leftmost instance of a part that equals $j-1$. Since $k\leq m$, we are guaranteed that $\upl_j\upl_{j+1}^{k}(\mu)=\upl_{j+1}^{k}\upl_j(\mu)$.
\end{proof}
\begin{Lemma}\label{submissionimpliesanything}
Suppose $\beta\in B_{\alpha,n}$ and that $j\in NE(\beta/\!\!/\alpha)$. Let $w\in CRHW_n$ such that $w(\alpha)=\beta$. Define an element $w'\in CRHW_n$ as follows.
\begin{itemize}
\item If $j\in \leg(w)$, let $w'$ be the unique reverse hookword with same content as $w$ obtained by setting $\leg(w')=\leg(w)\setminus \{j\}$. \item If $j\notin \leg(w)$, let $w'$ be the unique reverse hookword with same content as $w$ obtained by setting $\leg(w')=\leg(w)\cup \{j\}$. 
\end{itemize}
Then $w'(\alpha)=\beta$.
\end{Lemma}
\begin{proof}
Since $j\in NE(\beta/\!\!/\alpha)$, we know that $j$ is not the maximum element of $\supp(\beta/\!\!/\alpha)$. Hence, either $j\in \leg(w)$ or $j\notin \leg(w)$. Thus, the word $w'$ is well-defined. Furthermore, by Lemma \ref{submissionjgeq2} we know that $j\geq 2$.

Let the number of $\upl_{j+1}$ in $w$ equal $c$ where $c\geq 1$. Let the largest common suffix of $w$ and $w'$ be $w_1$. All operators $\upl_k$ that appear in $w_1$ satisfy $k\leq j-1$. Let $\mu=w_1(\alpha)$. Since $j\in NE(\beta/\!\!/\alpha)$, we know that there are at least $c$ instances of a part of length $j$ to the left of the leftmost part of length $j-1$ in $\mu$. Thus $\upl_j\upl_{j+1}^m(\mu)=\upl_{j+1}^m\upl_j(\mu)$ for $0\leq m\leq c$ by Lemma \ref{commutativityconsecutive}. This combined with the fact that $\upl_j\upl_i=\upl_i\upl_j$ for $i-j\geq 2$ implies that $w(\alpha)=w'(\alpha)=\beta$, as required.
\end{proof}

\begin{Lemma}\label{cardinalityvalidhookwords}
If $\beta\in B_{\alpha,n}$, then the number of words $w\in CRHW_n$ such that $w(\alpha)=\beta$ equals $2^{\vert NE(\beta/\!\!/\alpha)\vert}$.
\end{Lemma}
\begin{proof}
Let $p$ be the maximum element of $\supp(\beta/\!\!/\alpha)$. Then we know that 
\begin{eqnarray}
\supp(\beta/\!\!/\alpha)=\{p\}\uplus E(\beta/\!\!/\alpha)\uplus SE(\beta/\!\!/\alpha)\uplus NE(\beta/\!\!/\alpha).
\end{eqnarray}
Let $w$ be an element of $CRHW_n$ such that $w(\alpha)=\beta$. Since we know that the number of instances of $\upl_k$ in $w$ is equal to the number of boxes in column $k$ of $\beta/\!\!/\alpha$, $w$ is completely determined by $\leg(w)$.

Now by Lemma \ref{marriedimpliesleg}, we have that every element of $E(\beta/\!\!/\alpha)$ belongs to $\leg(w)$. Lemma \ref{dominationimpliesnotinleg} implies that every element of $SE(\beta/\!\!/\alpha)$ does not belong to $\leg(w)$. As far as elements of $NE(\beta/\!\!/\alpha)$ are concerned, we know by Lemma \ref{submissionimpliesanything}, that it does not matter whether they belong to $\leg(w)$ or not, as the final shape $\beta$ is going to be the same. Thus, for every subset $X\subseteq NE(\beta/\!\!/\alpha)$, the word $w\in CRHW_n$ satisfying  
\begin{eqnarray}
\leg(w)=E(\beta/\!\!/\alpha)\uplus X\uplus \{p\}
\end{eqnarray}
has the property that $w(\alpha)=\beta$. Hence the number of such words is $2^{\vert NE(\beta/\!\!/\alpha)\vert}$ as claimed.
\end{proof}

\begin{Lemma}
Let $\beta\in B_{\alpha,n}$. Then
\begin{eqnarray*}
\displaystyle\sum_{\substack{w\in CRHW_n\\w(\alpha)=\beta}}(-1)^{\asc(w)}&=& (-1)^{n-1-\vert E(\beta/\!\!/\alpha)\vert}\delta_{0,\vert NE(\beta/\!\!/\alpha)\vert}
\end{eqnarray*}
where $\delta$ denotes the Kronecker delta function.
\end{Lemma}
\begin{proof}
Firstly, recall that $\asc(w)=\vert\arm(w)\vert-1$. This in turn implies that $\asc(w)=n-\vert\leg(w)\vert$. By Lemma \ref{marriedimpliesleg}, we know that $E(\beta/\!\!/\alpha)\subseteq \leg(w)$. Further, if $p$ is the maximum element of $\supp(w)$, then $p\in \leg(w)$ as well. Lemma \ref{cardinalityvalidhookwords} implies that any subset of $NE(\beta/\!\!/\alpha)$ can belong to $\leg(w)$, and Lemma \ref{dominationimpliesnotinleg} gives that no element of $SE(\beta/\!\!/\alpha)$ belongs to $\leg(w)$.

Thus, we have the following sequence of equalities.
\begin{eqnarray}
\displaystyle\sum_{\substack{w\in CRHW_n\\w(\alpha)=\beta}}(-1)^{\asc(w)}&=& \displaystyle\sum_{\substack{w\in CRHW_n\\w(\alpha)=\beta}}(-1)^{n-\vert\leg(w)\vert}\nonumber\\&=& \displaystyle\sum_{X\subseteq NE(\beta/\!\!/\alpha)}(-1)^{n-1-\vert E(\beta/\!\!/\alpha)\vert-\vert X\vert}\nonumber\\&=&(-1)^{n-1-\vert E(\beta/\!\!/\alpha)\vert}\displaystyle\sum_{X\subseteq NE(\beta/\!\!/\alpha)}(-1)^{\vert X\vert}\nonumber\\&=& (-1)^{n-1-\vert E(\beta/\!\!/\alpha)\vert}\delta_{0,\vert NE(\beta/\!\!/\alpha)\vert}
\end{eqnarray}
\end{proof}
What follows is essentially a restatement of the above lemma.
\begin{Corollary}
If $\beta\in B_{\alpha,n}$ then 
\begin{eqnarray*}
k_{\beta}=\left\lbrace \begin{array}{cc}0 & \vert NE(\beta/\!\!/\alpha)\vert\geq 1\\ (-1)^{n-1-\vert E(\beta/\!\!/\alpha)\vert} & \vert NE(\beta/\!\!/\alpha)\vert=0.\end{array}\right.
\end{eqnarray*}
\end{Corollary}

Now, consider the following set of compositions.
\begin{eqnarray*}
P_{\alpha,n}=\{\beta: \beta\in B_{\alpha,n}, \vert NE(\beta/\!\!/\alpha)\vert=0\}
\end{eqnarray*}
Before we state a Murnaghan-Nakayama rule in its final definitive form, we will need another short lemma.
\begin{Lemma}\label{lem:exponentequalsheight}
If $\beta/\!\!/\alpha$ is an nc border strip of size $n$, then $n-1-|E(\beta/\!\!/\alpha)|=ht(\beta/\!\!/\alpha)$.
\end{Lemma}
\begin{proof}
Suppose there are $k>0$ boxes in some row of $\beta/\!\!/\alpha$, then this row contributes $k-1$ to $|E(\beta/\!\!/\alpha)|$. Let the number of boxes in rows that contain at least one box be $k_1,k_2,\ldots,k_r$ for some positive integer $r$. Notice that $\sum_{j=1}^{r}(k_j-1)$ equals $n-r$. Thus we have that 
\begin{eqnarray}
n-1-|E(\beta/\!\!/\alpha)|&=&n-1-\displaystyle\sum_{j=1}^{r}(k_j-1)\nonumber\\&=&r-1.
\end{eqnarray}
Now notice that $r-1$ is precisely $ht(\beta/\!\!/\alpha)$.
\end{proof}
This given, we can write down a Murnaghan-Nakayama rule for noncommutative Schur functions as follows.
\begin{Theorem}\label{MurNakfinal}
Given a composition $\alpha$, 
\begin{eqnarray*}
\Psi_{n}\cdot\ncsa=\displaystyle\sum_{\beta\in P_{\alpha,n}}(-1)^{ht(\beta/\!\!/\alpha)}\ncsb.
\end{eqnarray*}
\end{Theorem}
\begin{Example}
Consider the computation of $\Psi_4\cdot \ncs_{(2,1,3)}$. We need to find all compositions $\beta\in B_{(2,1,3),4}$ satisfying $\vert NE(\beta/\!\!/(2,1,3))\vert=0$. We will start by listing all the elements of $B_{(2,1,3),4}$ where the boxes of $\beta/\!\!/\alpha$ will be shaded blue. In the list below, beneath every composition we have noted those reverse hookwords that act on $(2,1,3)$ to give the composition under consideration. For the sake of clarity, the occurrences of $\mathfrak{t}$ in the reverse hookwords have been suppressed. Thus, for example, $1121$ actually denotes $\mathfrak{t}_1\mathfrak{t}_1\mathfrak{t}_2\mathfrak{t}_1$.
$$
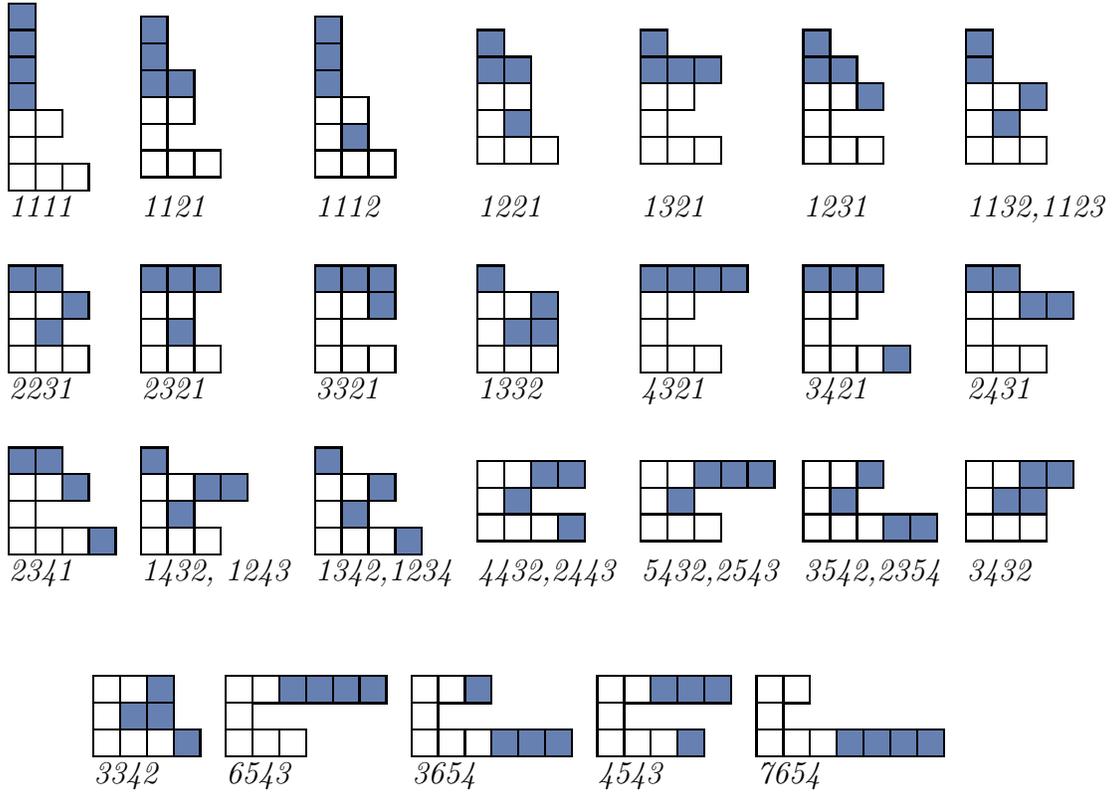
\ytableausetup{smalltableaux}
\begin{array}{lllllll}
$\begin{ytableau}
\myb\\\myb\\\myb\\\myb\\\cw &\cw\\\cw\\\cw &\cw &\cw
\end{ytableau}$ &
$\begin{ytableau}
\myb\\\myb\\\myb&\myb\\\cw &\cw\\\cw\\\cw &\cw &\cw
\end{ytableau}$ &
$\begin{ytableau}
\myb\\\myb\\\myb\\\cw &\cw\\\cw&\myb\\\cw &\cw &\cw
\end{ytableau}$ &
$\begin{ytableau}
\myb\\\myb&\myb\\\cw &\cw\\\cw&\myb\\\cw &\cw &\cw
\end{ytableau}$ &
$\begin{ytableau}
\myb\\\myb&\myb&\myb\\\cw &\cw\\\cw\\\cw &\cw &\cw
\end{ytableau}$ &
$\begin{ytableau}
\myb\\\myb&\myb\\\cw &\cw&\myb\\\cw\\\cw &\cw &\cw
\end{ytableau}$ &
$\begin{ytableau}
\myb\\\myb\\\cw &\cw&\myb\\\cw&\myb\\\cw &\cw &\cw
\end{ytableau}$\\
$1111$ & $1121$ & $1112$ & $1221$ & $1321$ & $1231$ & $1132,1123$\\
\text{} & \text{} &  \text{} & \text{} & \text{} & \text{} & \text{}\\
$\begin{ytableau}
\myb &\myb\\\cw &\cw&\myb\\\cw&\myb\\\cw &\cw &\cw
\end{ytableau}$&
$\begin{ytableau}
\myb&\myb&\myb\\\cw &\cw\\\cw&\myb\\\cw &\cw &\cw
\end{ytableau}$&
$\begin{ytableau}
\myb&\myb&\myb\\\cw &\cw&\myb\\\cw\\\cw &\cw &\cw
\end{ytableau}$&
$\begin{ytableau}
\myb\\\cw &\cw &\myb\\\cw&\myb&\myb\\\cw &\cw &\cw
\end{ytableau}$&
$\begin{ytableau}
\myb&\myb&\myb&\myb\\\cw &\cw\\\cw\\\cw &\cw &\cw
\end{ytableau}$ &
$\begin{ytableau}
\myb&\myb&\myb\\\cw &\cw\\\cw\\\cw &\cw &\cw&\myb
\end{ytableau}$&
$\begin{ytableau}
\myb&\myb\\\cw &\cw&\myb&\myb\\\cw\\\cw &\cw &\cw
\end{ytableau}$ \\
$2231$ & $2321$ & $3321$ & $1332$ & $4321$ & $3421$ & $2431$\\
\text{} & \text{} &  \text{} & \text{} & \text{} & \text{} & \text{}\\
$\begin{ytableau}
\myb&\myb\\\cw &\cw&\myb\\\cw\\\cw &\cw &\cw&\myb
\end{ytableau}$&
$\begin{ytableau}
\myb\\\cw &\cw&\myb&\myb\\\cw&\myb\\\cw &\cw &\cw
\end{ytableau}$&
$\begin{ytableau}
\myb\\\cw &\cw&\myb\\\cw&\myb\\\cw &\cw &\cw&\myb
\end{ytableau}$&
$\begin{ytableau}
\cw &\cw &\myb &\myb\\\cw&\myb\\\cw &\cw &\cw &\myb
\end{ytableau}$ &
$\begin{ytableau}
\cw &\cw &\myb &\myb&\myb\\\cw&\myb\\\cw &\cw &\cw
\end{ytableau}$&
$\begin{ytableau}
\cw &\cw&\myb\\\cw&\myb\\\cw &\cw &\cw&\myb &\myb
\end{ytableau}$ &
$\begin{ytableau}
\cw &\cw&\myb&\myb \\\cw&\myb&\myb\\\cw &\cw &\cw
\end{ytableau}$ \\
$2341$ & $1432, 1243$ & $1342,1234$ & $4432,2443$ & $5432,2543$ & $3542,2354$ & $3432$\\
\text{} & \text{} &  \text{} & \text{} & \text{} & \text{} & \text{}\\
\end{array}$$
$$
\ytableausetup{smalltableaux}
\begin{array}{lllllll}
$\begin{ytableau}
\cw &\cw&\myb \\\cw&\myb&\myb\\\cw &\cw &\cw &\myb
\end{ytableau}$&
$\begin{ytableau}
\cw &\cw&\myb&\myb&\myb&\myb \\\cw\\\cw &\cw &\cw
\end{ytableau}$&
$\begin{ytableau}
\cw &\cw&\myb \\\cw\\\cw &\cw &\cw&\myb &\myb &\myb
\end{ytableau}$&
$\begin{ytableau}
\cw &\cw&\myb&\myb &\myb\\\cw\\\cw &\cw &\cw &\myb
\end{ytableau}$&
$\begin{ytableau}
\cw &\cw \\\cw\\\cw &\cw &\cw&\myb&\myb&\myb&\myb
\end{ytableau}$&
\text{ }& \text{ }\\
$3342$ & $6543$ & $3654$ & $4543$ & $7654$ & \text{} & \text{ }\end{array}
$$
Now we list those compositions $\beta$ from the above list that satisfy $\vert NE(\beta/\!\!/\alpha)\vert\geq 1 $. These are
$$\ytableausetup{smalltableaux}
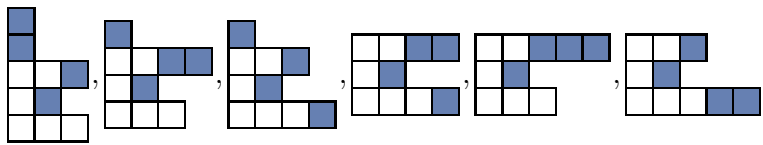
\begin{ytableau}
\myb\\\myb\\\cw &\cw&\myb\\\cw&\myb\\\cw &\cw &\cw
\end{ytableau}\hspace{0.5mm},
\begin{ytableau}
\myb\\\cw &\cw&\myb&\myb\\\cw&\myb\\\cw &\cw &\cw
\end{ytableau}\hspace{0.5mm},
\begin{ytableau}
\myb\\\cw &\cw&\myb\\\cw&\myb\\\cw &\cw &\cw&\myb
\end{ytableau}\hspace{0.5mm},
\begin{ytableau}
\cw &\cw &\myb &\myb\\\cw&\myb\\\cw &\cw &\cw &\myb
\end{ytableau}\hspace{0.5mm},
\begin{ytableau}
\cw &\cw &\myb &\myb&\myb\\\cw&\myb\\\cw &\cw &\cw
\end{ytableau}\hspace{0.5mm},
\begin{ytableau}
\cw &\cw&\myb\\\cw&\myb\\\cw &\cw &\cw&\myb &\myb
\end{ytableau}\hspace{0.5mm}.
$$
These are the elements of $B_{(2,1,3),4}$ that appear with coefficient $0$. The rest appear with either $+1$ or $-1$ depending on the parity of $ht(\beta/\!\!/\alpha)$. We will compute the coefficient of 
$$
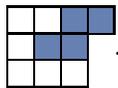
\ytableausetup{smalltableaux}
\begin{ytableau}
\cw &\cw&\myb&\myb \\\cw&\myb&\myb\\\cw &\cw &\cw
\end{ytableau}.$$
Since $ht((4,3,3)/\!\!/(2,1,3))=1$, the coefficient of $\ncs_{(4,3,3)}$ in $\Psi_{4}\cdot\ncs_{(2,1,3)}$ is $-1$. 

The complete expansion is thus as follows.
\begin{eqnarray*}
\Psi_{4}\cdot \ncs_{(2,1,3)}&=&-\ncs_{(1, 1, 1, 1, 2, 1, 3)} + \ncs_{(1, 1, 2, 2, 1, 3)}- \ncs_{(1, 1, 1, 2, 2, 3)} +
\ncs_{(1, 2, 2, 2, 3)}- \ncs_{(1, 3, 2, 1, 3)} + \ncs_{(1, 2, 3, 1, 3)}\\&& + \ncs_{(2, 3, 2, 3)} - \ncs_{(3, 2, 2, 3)} - \ncs_{(3, 3, 1, 3)}+\ncs_{(1, 3, 3, 3)}+ \ncs_{(4, 2, 1, 3)} - \ncs_{(3, 2, 1, 4)} - \ncs_{(2, 4, 1, 3)}\\&&+ \ncs_{(2, 3, 1, 4)}- \ncs_{(4, 3, 3)} + \ncs_{(3, 3, 4)} + \ncs_{(6, 1, 3)} - \ncs_{(3, 1, 6)}- \ncs_{(5, 1, 4)} +
\ncs_{(2, 1, 7)}  
\end{eqnarray*}
%Notice that there are 20 terms in the above expansion, corresponding precisely to the cardinality of $P_{(2,1,3),4}$.
\end{Example}
\bibliographystyle{amsalpha}

\begin{thebibliography}{10}

\bibitem{BSZ}
{\sc J.~Bandlow, A.~Schilling, and M.~Zabrocki},{\em {The Murnaghan-Nakayama rule for k-Schur functions}}, J. Combin. Theory Ser. A, 118 (2011), pp.~1588--1607.

\bibitem{BLvW}
{\sc C.~Bessenrodt, K.~ Luoto, and S.~van Willigenburg}, {\em {Skew quasisymmetric Schur functions and noncommutative Schur functions}}, Adv. Math., 226 (2011), pp.~4492--4532.

\bibitem{fomin-greene-1}
{\sc S.~Fomin and C.~Greene}, {\em {Noncommutative
  Schur functions and their applications}}, Discrete Math., 193 (1998),
  pp.~179--200.

%\bibitem{fulton-1}
%{\sc W.~Fulton}, {\em {Young Tableaux}}, {Cambridge University Press}, 1997.

\bibitem{GKLLRT}
{\sc I.~Gelfand, D.~Krob, A.~Lascoux, B.~Leclerc, V.~Retakh, and J.-Y. Thibon},
  {\em {Noncommutative symmetric functions}}, Adv. Math., 112 (1995),
  pp.~218--348.

\bibitem{HHL-1}
{\sc J.~Haglund, M.~Haiman, and N.~Loehr}, {\em {A combinatorial formula for
  Macdonald polynomials}}, J. Amer. Math. Soc., 18 (2005), pp.~735--761.

\bibitem{HLMvW-1}
{\sc J.~Haglund, K.~Luoto, S.~Mason, and S.~van Willigenburg}, {\em
  {Quasisymmetric {S}chur functions}}, J. Combin. Theory Ser. A, 118 (2011), pp.~463--490. 
%\bibitem{HLMvW-2}
%\leavevmode\vrule height 2pt depth -1.6pt width 23pt, {\em {Refinements of the
 % Littlewood-Richardson rule}},  Trans. Amer. Math. Soc., 363 (2011), pp.~1665--1686.

%\bibitem{LR-1}
%{\sc D.~Littlewood and A.~Richardson}, {\em Group characters and algebra},
 % Philosophical Transactions of the Royal Society of London. Series A,
  %Containing Papers of a Mathematical or Physical Character, 233 (1934),
  %pp.~99--141.
\bibitem{LMvW}
{\sc K.~Luoto, S.~Mykytiuk, and S.~van Willigenburg}, {\em { An introduction to quasisymmetric Schur functions. Hopf algebras, quasisymmetric functions, and Young composition tableaux.}}, Springer Briefs in Mathematics, Springer, New York, 2013.

\bibitem{macdonald-1}
{\sc I.G.~Macdonald}, {\em {Symmetric functions and Hall polynomials. 2nd ed.}},
  {Oxford University Press}, 1998.
  
 \bibitem{Murnaghan}
 {\sc F.D.~Murnaghan}, {\em {The characters of the symmetric group}}, Amer. J. Math., 59 (1937), pp.~739--753.
 
 \bibitem{Nakayama}
 {\sc T.~Nakayama}, {\em {On some modular properties of irreducible representations of a symmetric group.  I and II}}, Jap. J. Math., 17 (1940), pp.~165--184 and pp.~411--423.
\bibitem{sagan}
{\sc B.~Sagan}, {\em {The symmetric group. Representations, combinatorial
  algorithms, and symmetric functions. 2nd ed.}}, Springer, 2001.

\bibitem{stanley-ec2}
{\sc R.~Stanley}, {\em {Enumerative
  Combinatorics}}, vol.~2, Cambridge University Press, 1999.

\end{thebibliography}

\end{document}